\documentclass[11pt]{amsart}
\usepackage{amssymb,hyperref}
\usepackage{amsfonts}
\usepackage{enumerate}
\usepackage{amsmath, amsthm}
\usepackage{amscd}

\usepackage{graphicx}%
\hypersetup{pdfpagemode=FullScreen,colorlinks=true}

\input xy
\xyoption{all}

\newtheorem{thm}{Theorem}[section]

\newtheorem{Prop}[thm]{Proposition}
\newtheorem{lemma}[thm]{Lemma}

\newtheorem{co}[thm]{Corollary}

\theoremstyle{definition}

\newtheorem*{remark}{Remark}

\newcommand{\bq}{{\mathbb Q}}
\newcommand{\br}{{\mathbb R}}

\newcommand{\ra}{\rightarrow}

\begin{document}
\title[Proportionality principle]
{Proportionality principle for the simplicial volume of families of $\mathbb{Q}$-rank $1$ locally symmetric spaces}\author{Michelle Bucher, Inkang Kim and  Sungwoon Kim}
\date{}

\begin{abstract}
We establish the proportionality principle between the Riemannian volume and locally finite simplicial volume for $\mathbb{Q}$-rank $1$ locally symmetric spaces covered by products of hyperbolic spaces, giving the first examples for manifolds whose cusp groups are not necessarily amenable. Also, we give a simple direct proof of the proportionality principle for the locally finite simplicial volume and the relative simplicial volume of $\mathbb{Q}$-rank $1$ locally symmetric spaces with amenable cusp groups established by L\"oh and Sauer \cite{LoehSauerHilbert}.\end{abstract}

\thanks{M. Bucher gratefully acknowledges support from the Swiss National Science Foundation (grant  noPP00P2-128309/1). I. Kim gratefully acknowledges the partial support
of NRF grant  (R01-2008-000-10052-0) and a warm support of IHES and IHP
during his stay. M. Bucher and I. Kim are thankful to the  Mittag-Leffler Institute (Djursholm, Sweden) for its hospitality during the preparation of this paper.} \maketitle

\section{Introduction}

Gromov introduced the simplicial volume $\| M\| $ of closed manifolds in the beginning of the 80's and established the following general proportionality principle for closed manifolds.

\begin{thm}[Gromov-Thurston \cite{Gromov,Th78}] Let $(M,g)$ be a closed Riemannian manifold. Then there exists a constant $c(\widetilde{M},\widetilde{g})\in \mathbb{R}_{>0}\cup \{+\infty \}$ depending only on the universal cover $(\widetilde{M},\widetilde{g})$ such that
$$ \| M\| =\frac{\mathrm{Vol}(M,g)}{c(\widetilde{M},\widetilde{g})}.$$
\end{thm}

This fundamental principle, which gives an alternative for the Hirzebruch proportionality principle in odd dimension, allows, whenever $\| M\|\neq 0$, to consider the Riemannian volume of manifolds with a fixed universal cover as a topological invariant.

For symmetric spaces of noncompact type, the proportionality constant is easily seen to be equal to the sup norm $\|  \omega_{\widetilde{M}}\|_\infty \in \mathbb{R}_{>0} $ of the volume form  $\omega_{\widetilde{M}} \in H^n_c(\mathrm{Isom}(\widetilde{M}))$ (see \cite{Bucher}). The fact that in the compact case, the simplicial volume of locally symmetric spaces of non compact type is strictly positive  was established by Lafont and Schmidt (see also \cite{Sa82} and \cite{Bu07b} for the case of locally $\mathrm{SL}(3,\mathbb{R})/\mathrm{SO}(3)$ factors) answering a conjecture of Gromov. As a consequence the proportionality constant $\|  \omega_{\widetilde{M}}\|_\infty$ is finite for symmetric spaces $\widetilde{M}$ of noncompact type.

For open manifolds, the natural generalization of the simplicial volume, is the $\ell^1$-norm of the locally finite fundamental class of $M$, which we denote by $\| M\|_\mathrm{lf}$ (see Section \ref{section: Locally finite and relative homology} below for the definition). It is however in general not well behaved as far as the proportionality principle is concerned. Indeed, Gromov showed \cite{Gromov} that $\| M_1\times M_2 \times M_3 \|_\mathrm{lf}=0$ for any open manifolds $M_1,M_2,M_3$. L\"oh and Sauer generalized this vanishing result for symmetric spaces by showing that $\| M \|_\mathrm{lf}=0$ for any $\bq$-rank $\geq 3$ locally symmetric space \cite{LoehSauerDegreeThms}. On the positive side, Gromov and Thurston noted that the proportionality principle holds for the locally finite simplicial volume of complete hyperbolic manifolds of finite volume. This result was extended by L\"oh and Sauer to symmetric spaces with amenable cusp fundamental group \cite{LoehSauerHilbert}, which includes all  $\mathbb{R}$-rank $1$ symmetric spaces and  covers some $\mathbb{Q}$-rank $1$ symmetric spaces, such as Hilbert modular varieties. A further generalization to Riemannian manifolds with amenable boundary has just been provided by the third author and Kuessner \cite{Kim-Kuessner}.

L\"oh and Sauer's proof is based on a discrete approximation of smearing in measure homology and on Gromov's equivalence theorem \cite[page 57]{Gromov}, which has just recently admitted a straightforward (and complete) proof \cite{BBFIPP}. The  proof in \cite{Kim-Kuessner}  builds on Gromov's theory of multicomplexes \cite[Section 3]{Gromov} and the duality between $\ell^1$-homology and bounded cohomology. We propose an alternative complete self contained proof based on standard techniques from bounded cohomology. We recall the precise statement of L\"oh and Sauer's result which involves also the relative simplicial volume whose definition is given in Section \ref{section: Locally finite and relative homology}:

\begin{thm}[\cite{LoehSauerHilbert}] \label{thm:PropPrincLocFin}Let $V$ be a compact manifold with boundary $\partial V$ such that its interior $M=\mathrm{Int}(V)$ is a
 complete finite volume locally symmetric space. If the fundamental group of any connected component of $\partial V$ is amenable, then
$$\| V,\partial V \|=\| M\|_\mathrm{lf}=\frac{\mathrm{Vol}(M)}{\| \omega_{\widetilde{M}}\| _\infty}.$$
\end{thm}

\begin{co}[\cite{LoehSauerHilbert}] Let $M$ be as in the theorem. Then
$$\| M\|_\mathrm{lf}>0.$$
\end{co}

Other important consequences include the existence of degree theorems or positivity of the minimal volume (see \cite{LoehSauerHilbert}).

Note that in the case of a hyperbolic manifold $V$ of dimension $\geq 3$ with geodesic boundary $V$, Jungreis \cite{Jungreis} and Kuessner \cite{Kuessner} showed the strict inequality $\| V,\partial V \| <\frac{ \mathrm{Vol}(V)}{ \| \omega_{\mathbb{H}^n}\|_\infty}$. (See also \cite{FrigerioPagliantini} for a proof that this strict inequality is optimal.)

The proportionality principle in Theorem \ref{thm:PropPrincLocFin} is established only for the case that a compact manifold $V$ has amenable ends.
The following theorem implies that the proportionality principle may hold even though the fundamental group of a connected component of $\partial V$ is not amenable.

\begin{thm}\label{thm:1.4} Let $M$ be a $\mathbb{Q}$-rank $1$ locally symmetric $n$-space covered by a product of $\br$-rank $1$ symmetric spaces. Let $H$ be  the isometry group of the universal cover $\widetilde{M}$ of $M$. Suppose that the comparison map $c : H^n_{c,b}(H,\br_\varepsilon) \rightarrow H^n_c(H,\br_\varepsilon)$ is an isomorphism. Then
$$\| M\|_\mathrm{lf}=\frac{\mathrm{Vol}(M)}{\| \omega_{\widetilde{M}}\| _\infty}.$$
\end{thm}

The cohomology groups will be introduced in Section \ref{section: Cont coho}. Note that the comparison map is known to be surjective in top degree, while its injectivity is only conjecturally true. When $\widetilde{M}$ is the real hyperbolic $n$-space ($n\geq 2$), the injectivity was established in \cite{BBI}. Since the proof of the validity of the K\"unneth formula passes to bounded cohomology, the following corollary is immediate:

\begin{co}\label{co: hyp prod}Let $M$ be a $\mathbb{Q}$-rank $1$ locally symmetric $n$-space covered by a product of real hyperbolic spaces $\mathbb{H}^{n_1}\times ... \times\mathbb{H}^{n_k}$ with $n_1,...,n_k\geq 2$.  Then
$$\| M\|_\mathrm{lf}=\frac{\mathrm{Vol}(M)}{\| \omega_{\mathbb{H}^{n_1}\times ... \times \mathbb{H}^{n_k}}\| _\infty}.$$
\end{co}

Theorem \ref{thm:1.4} and Corollary \ref{co: hyp prod} in particular hold for reducible $\mathbb{Q}$-rank $1$ locally symmetric spaces. Note that a reducible $\mathbb{Q}$-rank $1$ locally symmetric space is a product of an irreducible $\mathbb{Q}$-rank $1$ locally symmetric space and closed locally symmetric spaces. Hence, it is easy to see that each component of its ends does not have amenable fundamental group.
To our knowledge, this gives the first class of examples of $\mathbb{Q}$-rank $1$ locally symmetric spaces for which the proportionality principle is known without the assumption that the fundamental group of each component of their ends is amenable.

\subsection*{On the proofs} The idea of the proof of Theorem \ref{thm:PropPrincLocFin} is as follows. Let $\omega_M\in H^n_{cpct}(M)$ be the top dimensional compact support cohomology class which equals to the volume of $M$ when evaluated on the locally finite fundamental class $[M]_\mathrm{lf}\in H_n^\mathrm{lf}(M)$ (see Sections \ref{section: Locally finite and relative homology} and \ref{section: Cohomology with compact support} for more details). Two seminorms $\| \omega_M \|_\infty$ and $\| \omega_M \|^\infty$ are defined on $H^n_{cpct}(M)$ and it is a simple consequence of the Hahn-Banach theorem (see \cite{Gromov}, \cite{Loeh} or Theorems \ref{thm:dualityFORopen} and \ref{thm:dualityFORrelative} here) that
$$ \| M\|_\mathrm{lf}=\frac{\mathrm{Vol}(M)}{\| \omega_{{M}}\| ^\infty} \ \ \mathrm{and} \ \  \| V,\partial V \|=\frac{\mathrm{Vol}(M)}{\| \omega_{{M}}\| _\infty}.$$

To prove Theorem \ref{thm:PropPrincLocFin} we hence just need to show that
$$\| \omega_{{M}}\| _\infty = \| \omega_{{M}}\| ^\infty = \| \omega_{\widetilde{M}}\| _\infty,$$
as stated in Proposition \ref{mainProp}. In case $M$ is closed, the compact support cohomology of $M$ is nothing else than the usual singular cohomology, the norms $\| \omega_{{M}}\| _\infty $ and $ \| \omega_{{M}}\| ^\infty $ are easily seen to be equal and the equality  $\| \omega_{{M}}\| _\infty = \| \omega_{\widetilde{M}}\| _\infty$ follows from the fact that the top dimensional isomorphism
$$H^n_c(\mathrm{Isom}(\widetilde{M}))\rightarrow H^n(\pi_1(M))\cong H^n(M)$$
is isometric (see \cite{Bucher} for more details on the compact case).

If $M$ is noncompact, then $H^n(\pi_1(M))=H^n(M)=0$ and the above approach fails. We choose to pass instead to bounded cohomology. We introduce transfer maps
\[
\begin{CD}
  H^n_{cpct, b}(M)    @>trans_b>>          H^n_{c,b}(G,\br) \\
                      @VV c V                             @VV c V \\
  H^n_{cpct}(M) @>trans>>  H^n_{c}(G,\br)    \end{CD} \]
directly on the compact support cohomology in such a way that the unbounded transfer map agrees with the de Rham transfer (see Section \ref{Section: Transfer}) and hence sends the volume class $\omega_M$ to $\omega_{\widetilde{M}}$. Our proof amounts to showing that the transfer map
\[
\begin{CD}
  trans:H^n_{cpct}(M) @> >>  H^n_{c}(G,\br)    \end{CD} \]
  is an isometric isomorphism in top dimension.

Transfer maps in cohomology are the dual analogue of smearing in homology, but the cohomology approach has the striking advantage that it does not rely on the technical and difficult fact that measure homology is isometrically isomorphic to $\ell^1$-homology. Instead, all our proofs are straightforward and self contained.

The proof of Theorem \ref{thm:1.4} is a development of ideas of \cite{Kim-Kim}, where the positivity of the simplicial volume for $\bq$-rank $1$ lattices lying in the product of $\br$-rank $1$ simple groups is proven. The injectivity of the comparison map allows to conclude that $\| \omega_{{M}}\| ^\infty \leq \| \omega_{\widetilde{M}}\| _\infty$ (compare with the inequality $\| \omega_{{M}}\| ^\infty \leq \|\Theta \| _\infty<+\infty$, where $\Theta$ is one particular bounded cocycle representing $ \omega_{\widetilde{M}}$ established in \cite{Kim-Kim} for the mere positivity of the simplicial volume), and the other inequality is proven in Section \ref{sec:5.3} for the proof of Theorem \ref{thm:PropPrincLocFin}.

\subsection*{Structure of the paper} We recall some basics of locally finite and relative homology in Section \ref{section: Locally finite and relative homology}, cohomology with compact support in Section \ref{section: Cohomology with compact support}  and continuous group cohomology in Section \ref{section: Cont coho}. We introduce our new transfer maps in Section \ref{Section: Transfer} and establish there a few simple properties, in particular the first inequality of Proposition \ref{mainProp}. In Section \ref{Section: coning}, we prove the second inequality of Proposition \ref{mainProp}  by exhibiting a norm nonincreasing cohomology map $f^*:H^*_b(\pi_1(M))\rightarrow H^*_{cpct,b}(M)$ induced by a cusp map. This map should be of independent interest. Finally, the proof of Theorem \ref{thm:1.4}  is given is section \ref{section: product of rank 1}

\section{Locally finite and relative homology}\label{section: Locally finite and relative homology}

Let $M$ be a manifold.
Borel and Moore \cite{Borel-Moore} introduce a homology theory for locally compact spaces, called the Borel-Moore homology. There are several ways to describe this homology theory and we will define it here as the homology of the locally finite chain complex of $M$ and call this the locally finite homology of $M$.

Let $R$ be a field of characteristic zero.
The locally finite chain complex $C^\mathrm{lf}_*(M,R)$ of $M$ with coefficients in $R$ is defined by the chain complex of infinite singular chains $c=\sum_{i=0}^\infty a_i \sigma_i$ where $\sigma_i$ is a singular simplex, $a_i \in R$ and the sum is locally finite in the following sense: Any compact subset of $M$ intersects the image of only finitely many singular simplices occurring in $c$.
The usual boundary map $\partial$ on the singular chains is well-defined on $C^\mathrm{lf}_*(M,R)$. Then,
the locally finite homology $H^\mathrm{lf}_*(M,R)$ with coefficients in $R$ is defined by
$$H^\text{lf}_* (M,R)=H_*(C^\mathrm{lf}_*(M,R), \partial).$$

For compact manifolds, the locally finite homology coincides with the usual singular homology, but
it gives useful homology groups for noncompact manifolds. An essential  advantage of locally finite homology as opposed to singular homology is the existence of a fundamental class of any oriented manifold. Indeed, the $n$-th singular homology of a noncompact manifold vanishes, and hence cannot contain any fundamental class. The existence of a well-defined fundamental class $[M]\in H^\mathrm{lf}_n(M,\mathbb{R})$ for any oriented $n$-dimensional manifold is established via the Poincar\'e duality
$$H^\mathrm{lf}_i(M,R) \cong H_{cpct}^{n-i}(M,R)$$
with cohomology with compact support, a cohomology theory which will be briefly recalled in the next section. The duality implies that $H^\mathrm{lf}_n(M,\mathbb{R}) \cong H_{cpct}^{0}(M,\mathbb{R})\cong \mathbb{R}$ and the existence of a canonical fundamental class $[M]\in H^\mathrm{lf}_n(M,\mathbb{R})$. We refer the reader to \cite{Borel-Moore} and \cite{Bredon} for more detailed explanations about locally finite homology.

The simplicial $\ell^1$-seminorm in $C^\mathrm{lf}_*(M,\mathbb{R})$ is defined by setting $\| c\|_1 =\sum_{i=0}^\infty |a_i|$ for a chain $c=\sum_{i=0}^\infty a_i \sigma_i$ in $C^\mathrm{lf}_*(M,\mathbb{R})$. This norm gives rise to a seminorm on the locally finite homology $H^\mathrm{lf}_*(M,\mathbb{R})$ as follows:
$$ \| \alpha \|_1 = \inf_{z} \|z \|_1,$$
where the infimum is taken over all locally finite cycles representing $\alpha \in H^\mathrm{lf}_*(M,\mathbb{R})$. For an oriented $n$-manifold $M$, the simplicial volume $\| M \|_\mathrm{lf}$ of $M$ is defined as the seminorm of its fundamental class $[M]\in H^\mathrm{lf}_n(M,\mathbb{R})$.

Now, let $V$ be a compact $n$-manifold with boundary $\partial V$. For a relative chain $c$ in the relative singular chain complex $C_*(V,\partial V,\mathbb{R})$, the simplicial $\ell^1$-norm is given by the infimum of the $\ell^1$-norms of its representatives. This norm induces a seminorm on the relative singular homology $H_*(V,\partial V,\mathbb{R})$ as in the case of locally finite homology. Then, the relative simplicial volume $\| V,\partial V \|$ is defined as the seminorm of its relative fundamental class $[V,\partial V] \in H_*(V,\partial V,\mathbb{R})$.

Let $M$ be the interior of $V$. Suppose that $M$ is an oriented manifold. We can obtain the inequality $\| V,\partial V \| \leq \| M\|_\mathrm{lf}$ as follows: Consider a compact submanifold $V_0$ of $M$ by removing a collared neighborhood of $\partial V$. It is clear that $V_0$ is homeomorphic to $V$. Let $c=\sum_{i=0}^\infty a_i \sigma_i$ be any locally finite cycle representing the fundamental class $[M] \in H^\mathrm{lf}_n(M,\mathbb{R})$. It is a standard fact that $c|_{V_0}=\sum_{im \sigma_i \cap V_0 \neq \emptyset} a_i \sigma_i$
represents the relative fundamental class $[V, V-V_0] \in H_n(V,V-V_0,\mathbb{R})$.
Because $V-V_0$ is the collared neighborhood of $\partial V$, one can think of $c|_{V_0}$ as a cycle representing
$[V,\partial V] \in H_n(V,\partial V)$. Clearly, we have an inequality
$$ \| c|_{V_0} \|_1 \leq \|c\|_1. $$
Taking the infimum over all locally finite cycles $c$ representing $[M]$, we have
$$ \|V,\partial V \| =\inf_{c'} \|c' \|_1 \leq  \inf_{c} \| c|_{V_0} \|_1 \leq \inf_{c} \|c\|_1 = \|M\|_\mathrm{lf},$$
where $c'$ runs over all relative cycles representing $[V,\partial V]$.

\section{Cohomology with compact support}\label{section: Cohomology with compact support}
\subsection{Singular cohomology}\label{section:Singular cohomology}
Let $M$ be an oriented $n$-manifold.
A singular cochain $f : C_q(M,\br) \rightarrow \br$ is said to be a cochain with compact support if there exists a compact subset $K$ such that $f(\sigma)=0$ whenever $\mathrm{Im}(\sigma)\cap K=\emptyset $.
The cohomology with compact supports of $M$, denoted by $H^*_{cpct}(M,\br)$, is defined as the cohomology of the subcocomplex $C^*_{cpct}(M,\br)$ of cochains with compact support in $C^*(M,\br)$ with the usual coboundary operator $\delta$ on the singular cochain complex. For more details, see \cite[Chapter 3]{Ha02}.

Consider the sup norm with respect to singular simplices, that is, for a cochain $f\in C^q_{cpct}(M,\br)$, set
$$\| f\|_\infty = \sup_{\sigma} |f(\sigma)|,$$
where $\sigma :\Delta^q \rightarrow M $ runs over all singular $q$-simplices.

Noting that $\| \delta(f)\|_\infty \leq (q+2)\cdot \| f\|_\infty$, we can consider the subcocomplex of bounded cochains with compact support
$$C^q_{cpct,b}(M,\br)=\{f\in C^q_{cpct}(M,\br)\mid \| f\|_\infty <+\infty \}$$
which gives rise to the bounded cohomology with compact support, denoted by $H^*_{cpct,b}(M,\br)$. The inclusion of cocomplexes induces a cohomology homomorphism
$$c:H^*_{cpct,b}(M,\br)\longrightarrow H^*_{cpct}(M,\br)$$
which is traditionally named {\it comparison map}.

The sup norm induces a seminorm $\| \cdot \|_\infty$ on cohomology, both in the bounded and unbounded case.
For $\beta \in H^q_{cpct}(M,\br)$, its seminorm is defined by
$$ \|\beta\|_\infty = \inf\{\|b\|_\infty \mid b\in C^q_{cpct}(M,\br), \  \delta b=0, \ [b]=\beta\}$$
where $[b]$ as usual denotes the cohomology class represented by the cocycle $b$.
A seminorm $\| \cdot \|_\infty$ is defined on $H^q_{cpct,b}(M,\br)$ as above.
Note that
$$ \|\beta\|_\infty=\inf\{\|\beta_b\|_\infty \mid \beta_b\in H^q_{cpct,b}(M,\br), \ c(\beta_b)=\beta\}.$$

\subsection{De Rham cohomology}\label{derham}

Let $M$ be a smooth manifold. Let $\Omega_{cpct}^q(M)$ denote the set of all differential $q$--forms on $M$ with compact support. The de Rham cohomology of $M$ with compact support, denoted by $H^*_{dR,cpct}(M,\mathbb{R})$, is defined as the cohomology of the cocomplex $\Omega_{cpct}^*(M)$ with the standard exterior derivative. By \cite{deRham}, the cohomology $H^*_{cpct}(M,\br)$ with compact support can be computed as $H^*_{dR,cpct}(M,\mathbb{R})$.

It is a standard fact that there is no difference between singular and smooth singular cohomology (with compact support).  Hence, we will use smooth singular cohomology with compact support to describe the de Rham isomorphism between de Rham cohomology and singular cohomology with compact support.
To denote smooth singular cohomology with compact support, we use the same notations as in Section \ref{section:Singular cohomology}.
For $\omega \in \Omega_{cpct}^q(M)$ and a smooth $q$-simplex $\sigma : \Delta^q \rightarrow M$, define a map $$I(\omega)(\sigma)=\int_\sigma \omega := \int_{\Delta^q} \sigma^* \omega,$$
and extend this linearly on $C_q(M,\br)$. Then, it can be easily seen that $I : \Omega_{cpct}^*(M) \rightarrow C_{cpct}^*(M,\br)$ is well-defined and a cochain map. Furthermore, this cochain map induces an isomorphism
$I : H^*_{dR,cpct}(M,\br) \rightarrow H^*_{cpct}(M,\mathbb{R})$ in cohomology.
This is known as the de Rham theorem in the case that $M$ is a closed manifold. We refer the reader to \cite[Chapter IV]{Demaily} or \cite[Chapter 16]{Lee} for further details.

Now, suppose that $M$ is a $\mathbb{Q}$-rank $1$ locally symmetric space.
Note that the geodesic straightening map $str : C^\mathrm{lf}_*(M,\mathbb{R}) \rightarrow C^\mathrm{lf}_*(M,\mathbb{R})$ is well-defined and moreover, chain homotopic to the identity \cite{Kim-Kim}.
Hence, it can be seen that if $f$ is a $q$-cochain with compact support, $f\circ str : C_q(M,\br) \rightarrow \br$ is also a cochain with compact support.
This allows us to have a map $C^*_{cpct}(M,\br) \rightarrow C^*_{cpct}(M,\br)$ defined by $f \mapsto f\circ str$. Furthermore, this map induces the identity map in cohomology since the geodesic straightening map $str$ is chain homotopic to the identity. Thus, the de Rham isomorphism $I : H^*_{dR,cpct}(M,\br) \rightarrow H^*_{cpct}(M,\mathbb{R})$ can also be induced, at cochain level, by the map
$$I_{str}(\omega)(\sigma)=I(\omega)(str(\sigma))=\int_{str(\sigma)} \omega.$$
Observe that for arbitrary manifolds with a given straightening, the map $I_{str}$ is not well defined on the compact support cohomology, as the straightening will not map locally finite chains to locally finite chains in general. 

\subsection{Dual norms}

There is a natural pairing
$$\langle \ , \ \rangle:H^q_{cpct}(M,\br)\otimes H^\mathrm{lf}_q(M,\mathbb{R})\longrightarrow \br$$
 between cohomology with compact support and locally finite homology defined for $\beta=[b]\in H^q_{cpct}(M,\br)$ and $\alpha=[z]\in H_q^\mathrm{lf}(M,\mathbb{R})$ as
 $$\langle \beta , \alpha \rangle=\langle [b],[z]\rangle:= b(z).$$

It is straightforward to check that this definition is independent of the choices of cocycles and cycles representing $\beta$ and $\alpha$ respectively. It is now easy to show that
\begin{equation}  | \langle \beta , \alpha \rangle|\leq \| \beta\|_\infty \cdot \| \alpha \|_1,\label{equ:dualityInequalityFORnorm_}\end{equation}
but equality does not hold in general even for $\alpha=[M]$. Examples are given below. In order to get a duality of norms in top dimension, Gromov introduced \cite[page 17]{Gromov} a different norm on cohomology with compact
support as follows : Let $\mathcal{S}_q$ be the set of singular $q$-simplices $\Delta^q \rightarrow M$.
A subset $\Phi$ of $\mathcal{S}_q$ is said to be locally finite if any compact subset of $M$ intersects the image of only finitely many elements of $\Phi$.
Let $\mathcal{S}_q^\mathrm{lf}$ denote the set of all locally finite subsets of $\mathcal{S}_q$.

For a $q$-cochain $b$ with compact support in $C^q_{cpct}(M,\br)$, define a seminorm $\| b \|_\Phi$ by setting $$\| b \|_\Phi = \sup_{\sigma \in \Phi} |b(\sigma)|,$$
for each $\Phi \in \mathcal{S}^\mathrm{lf}_q$. Subsequently, obtain a seminorm $\| \cdot \|_\Phi$ on $H^q_{cpct}(M,\br)$ and take the supremum of these seminorms over all locally finite subsets of $\mathcal{S}_q$.
Then, we have a new seminorm $\| \cdot \|^\infty$ on $H^q_{cpct}(M,\br)$ defined by
$$\|\beta \| ^\infty = \sup_{\Phi \in \mathcal{S}^\mathrm{lf}_q} \| \beta \|_\Phi=\sup_{\Phi \in \mathcal{S}^\mathrm{lf}_q}\inf_{\beta=[b]}\sup_{\sigma\in \Phi}|b(\sigma)|.$$

Observe that for any $\Phi \in \mathcal{S}^\mathrm{lf}_q$, any $\sigma\in \Phi$ and $b\in C^q_{cpct}(M,\br)$, we have
$$ \sup_{\sigma\in\Phi}|b(\sigma)|\leq \|b\|_\infty.$$
Hence it follows that
\begin{equation}
\left\Vert
\beta\right\Vert ^{\infty}\leq\left\Vert \beta\right\Vert
_{\infty}.\label{equ:norm^LEQnorm_}
\end{equation}

In general the above inequality is not an equality. Take $V=[0,1]$
and $M=(0,1)$ with a standard Euclidean metric. Let $\omega_M\in
H^1_{cpct}(M)$ be the unique cohomology class with $\langle
\omega_M, [M]\rangle=1$. It is not difficult to show that
$\|M \|_\mathrm{lf}=\infty$, whereas $\| V,\partial V \| =1<\infty$. In view of
Theorems \ref{thm:dualityFORopen} and
\ref{thm:dualityFORrelative}, $||\omega_M||^\infty=0$ and
$||\omega_M||_\infty>0$ for $M=(0,1)$.

\begin{lemma} \label{Lemma:dualityInequalityFORnorm^} Let $\beta \in H^q_{cpct}(M,\br)$ and $\alpha \in H_q^\mathrm{lf}(M,\mathbb{R})$. Then
 $$| \langle \beta , \alpha \rangle |\leq \|\beta\|^\infty \cdot \|\alpha \|_1. $$
\end{lemma}

\begin{proof}
Let $f$ be a representative of $\beta$ and $z=\sum_{\sigma \in \Phi} a_\sigma \sigma$ be a representative of $\alpha$. Then,
$$ | \langle \beta , \alpha \rangle |
= | \langle f , z \rangle |
= \Big|\sum_{\sigma \in \Phi} a_\sigma \cdot f(\sigma) \Big|
\leq \| f \|_\Phi \cdot \sum_{\sigma \in \Phi} | a_\sigma | = \| f \|_\Phi \cdot \| z \|_1.$$
Taking the infimum over all representatives of $\beta$ and then taking the supremum over all $\Phi$,
$$ | \langle \beta , \alpha \rangle |
\leq \| \beta \|_\Phi \cdot \| z \|_1
\leq \| \beta \|^\infty \cdot \| z \|_1.$$
The desired inequality follows by taking the infimum over all representatives of $\alpha$.
\end{proof}

Note that in view of (\ref{equ:norm^LEQnorm_}) the inequality (\ref{equ:dualityInequalityFORnorm_}) immediately follows from the lemma. Gromov \cite{Gromov} showed that the simplicial volume of a non-compact manifold can be computed in terms of this norm $\| \cdot \|^\infty$ of the dual class
in cohomology with compact support as follows:

\begin{thm}\label{thm:dualityFORopen} Let $M$ be an oriented $n$-manifold, and let $0\neq \beta\in H^n_{cpct}(M,\br)$. Then,
 $$\frac{| \langle \beta , [M] \rangle |}{\|\beta\|^\infty}= \|M\|_\mathrm{lf}. $$
\end{thm}

The original sup norm $\| \beta \|_\infty$ is still of interest, as it turns out to be dual to the $\ell^1$-norm of manifolds $V$ with boundary $\partial V$, in the case where $M$ is the interior of $V$.

\begin{thm}\label{thm:dualityFORrelative} Let $V$ be a compact, oriented $n$-manifold with boundary $\partial V$,
and let $M$ denote its interior. Let $0\neq \beta\in
H^n_{cpct}(M,\br)$. Then,
 $$\frac{| \langle \beta , [M] \rangle |}{\|\beta\|_\infty}= \|V,\partial V \|. $$
 \end{thm}

For detailed proofs of Theorem \ref{thm:dualityFORopen} and \ref{thm:dualityFORrelative},
we refer the reader to \cite{Loeh}.
Observe that Theorems \ref{thm:dualityFORopen} and \ref{thm:dualityFORrelative}, together with the simple inequality (\ref{equ:norm^LEQnorm_}) immediately reproves the inequality
 \[ \left\Vert V,\partial V\right\Vert \leq\left\Vert
M\right\Vert_\mathrm{lf} ,\]
when $M$ is the interior of $V$.

Let $\omega_M\in H^n_{cpct}(M,\br)$ denote the unique cohomology class defined by $\langle \omega_M,[M]\rangle=\mathrm{Vol}(M)$. In view of Theorems \ref{thm:dualityFORopen} and \ref{thm:dualityFORrelative}, our proportionality principles in Theorem \ref{thm:PropPrincLocFin}  will immediately follow from:
\begin{Prop}\label{mainProp}There is an equality of norms
$$ \| \omega_M \|_\infty = \| \omega_M \|^\infty = \| \omega_{\widetilde{M}}\|_\infty.$$
\end{Prop}
Given the inequality $ \| \omega_M \|_\infty \geq  \| \omega_M \|^\infty $ established above (\ref{equ:norm^LEQnorm_}), to prove the proposition, and hence Theorem  \ref{thm:PropPrincLocFin} it suffices to prove the inequalities
\begin{itemize}
\item $  \| \omega_M \|^\infty \geq \| \omega_{\widetilde{M}}\|_\infty$; this will be proven at the end of Section \ref{Section: Transfer},
\item $ \| \omega_M \|_\infty \leq \| \omega_{\widetilde{M}}\|_\infty$; this will be proven at the end of Section \ref{Section: coning}.
\end{itemize}

\section{Continuous cohomology}\label{section: Cont coho}

Let $G$ be a topological group. Recall that the continuous cohomology $H^*_c(G,\mathbb{R})$ of $G$ (with trivial $\mathbb{R}$ coefficients) is the cohomology of the cocomplex
$$C^{q+1}(G,\mathbb{R})=\{ f: G^{q+1}\rightarrow \mathbb{R} \mid f \mathrm{ \ continuous \,and \,  } G\mathrm{-invariant} \},$$
endowed with its homogeneous coboundary operator $\delta:C^{q+1}(G,\mathbb{R})\rightarrow C^{q+2}(G,\mathbb{R})$. The subcocomplex of bounded functions
$$C^{q+1}_b(G,\mathbb{R})=\{ f\in C^{q+1}(G,\mathbb{R}) \mid \| f \|_\infty = \sup_{g_0,...,g_q\in G} |f(g_0,...,g_q)|<+\infty   \}$$
leads to the continuous bounded cohomology $H^q_{c,b}(G,\mathbb{R})$ of $G$. The inclusion of cocomplexes $C^{q+1}_b(G,\mathbb{R})\subset C^{q+1}(G,\mathbb{R})$ induces a cohomology map $H^q_{c,b}(G,\mathbb{R})\rightarrow H^q_c(G,\mathbb{R})$. The sup norm defines a cohomology seminorm as
$$\| \beta \|_\infty = \inf \{ \| f\|_\infty \mid [f]=\beta, \ f\in C^{q+1}_b(G,\mathbb{R})\},$$
for $\beta $ in $H^q_c(G,\mathbb{R})$ or $H^q_{c,b}(G,\mathbb{R})$. Note that in the case where $\beta\in H^q_c(G,\mathbb{R})$ cannot be represented by a bounded cocycle, the infimum over the empty set is defined as $\| \beta \|_\infty = +\infty $.

If the group $G$ is endowed with the discrete topology, then the continuity condition is void and we remove the subscript ``c" from the notation. Note that one then recovers the Eilenberg-MacLane group cohomology.

Many different cocomplexes can be used to compute these cohomology groups. An important example is the following: Let $K$ be a maximal compact subgroup of $G$. The cohomology of the cocomplex of $G$-invariant continuous functions on the product of $q+1$ copies of $G/K$,
$$C((G/K)^{q+1},\mathbb{R})^G=\{ f: (G/K)^{q+1}\rightarrow \mathbb{R} \mid f \mathrm{ \ continuous \, and \,  } G \mathrm{-invariant} \},$$
endowed with its homogeneous coboundary operator is isomorphic to the continuous cohomology of $G$ \cite[Ch. III, Prop. 2.3]{Guichardet}. Similarly, the subcocomplex of bounded functions on $(G/K)^{q+1}$,
$$C_b((G/K)^{q+1},\mathbb{R})^G=\{ f\in C((G/K)^{q+1},\mathbb{R})^G  \mid \| f \|_\infty <+\infty  \},$$
computes the continuous bounded cohomology of $G$ \cite[Cor. 7.4.10]{Monod}. Furthermore, the sup norm on these cocomplexes induces a seminorm on cohomology which agrees with the above defined seminorm \cite[Cor. 7.4.10]{Monod}.

For Theorem \ref{thm:1.4} we will need also nontrivial coefficients, namely we will consider actions of $G$ on $\br$ by multiplication by $\pm 1$ given by a homomorphism $G\rightarrow \mathbb Z /2\mathbb Z$ (see the beginning of Section \ref{section: product of rank 1}). The corresponding cohomology groups are obtained by replacing $G$-invariant by $G$-equivariant cochains in the various cochain complexes.

Let now $G$ be a Lie group and $\Gamma <G$ a lattice. The inclusion $\Gamma <G$ induces restriction maps $C^{q+1}(G,\mathbb{R})\rightarrow C^{q+1}(\Gamma,\mathbb{R})$ and $C((G/K)^{q+1},\mathbb{R})^G\subset C((G/K)^{q+1},\mathbb{R})^\Gamma$ which in turn induces cohomology maps
$$H^*_{c}(G,\mathbb{R})\rightarrow H^*(\Gamma,\mathbb{R}) \mathrm{ \ and \ } H^*_{c,b}(G,\mathbb{R})\rightarrow H^*_b(\Gamma,\mathbb{R}).$$
If $\Gamma$ is cocompact, both maps admit a norm decreasing left inverse, while if $\Gamma$ is not cocompact, only the bounded cohomology map admits a norm decreasing left inverse. The left inverse is given by a transfer map
$$trans_\Gamma:C((G/K)^{q+1},\mathbb{R})^\Gamma \rightarrow C((G/K)^{q+1},\mathbb{R})^G $$
which is, for $c\in  C((G/K)^{q+1},\mathbb{R})^\Gamma $ and $g_0K,...,g_qK\in G/K$ defined as follows:
$$  trans_\Gamma(c)(g_0K,...,g_qK)=\int_{g\in D } c(gg_0K,...,gg_qK)d\mu(g),$$
where $D$ is any fundamental domain for $\Gamma \setminus G$ and the Haar measure $\mu$ is normalized so that $\mu(D)=1$. Note that $ trans_\Gamma(c)$ is finite if either $c$ is a bounded cochain or if the lattice $\Gamma <G$ is cocompact. As the transfer map commutes with the coboundary operator, it induces in these cases cohomology maps which are indeed left inverse to the restriction map. Finally observe that the transfer map does not increase seminorms.

\section{Transfer maps}\label{Section: Transfer}

\subsection{Transfer on de Rham cohomology}\label{deRham}

It is easy to define a transfer map $H^*_{cpct}(M)\rightarrow H^*_c(G)$ through the de Rham cohomology with compact support and the Van Est isomorphism. Indeed, at the cochain level, one defines
$$trans_{dR}:\Omega^q(G/K)_{cpct}^\Gamma \longrightarrow \Omega^q(G/K)^G$$
by sending the differential $q$-form $\alpha\in \Omega^q(G/K)_{cpct}^\Gamma$ to the form $\int_{g\in D} g^*\alpha d\mu(g)$, or more precisely to the form defined, for $x\in G/K$ and $V_1,...,V_q\in T_x(G/K)$, by
$$trans_{dR}(\alpha)_x(V_1,...,V_q)=\int_{g\in D} \alpha_{gx}(g_*(V_1),...,g_*(V_q)) d\mu(g).$$
As above, $D$ is any fundamental domain for $\Gamma\setminus G$ and the Haar measure $\mu$ is normalized so that the measure of $D$ is $1$. It is easy to check that this definition is independent of the fundamental domain $D$ and that the resulting differential form $trans_{dR}(\alpha)$ is $G$-invariant. Furthermore, the transfer map clearly commutes with the differential operator, and hence induces a cohomology map
$$trans_{dR}:H^*_{dR,cpct}(M)\longrightarrow  H^*(\Omega^*(G/K)^G)\cong \Omega^*(G/K)^G\cong H^*_c(G).$$
Naturally, we now want to understand the transfer map on singular cohomology and in particular on bounded singular cohomology, or in other words, to find transfer maps so that the following diagram commutes,
\begin{equation}
\begin{CD}  \label{equ: CD with trans}
  H^q_{cpct, b}(M)    @>trans_b>>          H^q_{c,b}(G,\br) \\
                      @VV c V                             @VV c V \\
  H^q_{cpct}(M) @>trans>>  H^q_{c}(G,\br)   \\
     @AA \cong A                             @AA \cong A \\
     H^q_{dR,cpct}(M) @>trans_{dR}>>  \Omega^q(G/K)^G.
   \end{CD}
   \end{equation}

This is the purpose of the next subsection. Note that the existence of these transfer maps would be easy to show for arbitrary manifolds, given that the singular cohomology with compact support can be computed on measurable cochains. We prefer to restrict to $\mathbb{Q}$-rank $1$ manifolds where the proof is straightforward.

\subsection{Transfer on singular (bounded) cohomology}\label{subsec: def of p}

As usual, when defining a transfer map, we would like to integrate the evaluation of a cochain $c\in C^q_{cpct}(M)$ on translates $g\cdot \sigma$ of a singular simplex $\sigma$ over a fundamental domain $g\in D$ for $\Gamma\setminus G$. This is only possible if the cochain presents certain regularity properties. However, at this point, the cochain $c$ is completely arbitrary. We will thus start by replacing $c$ by a better behaved cochain.

The $\mathbb{Q}$-rank $1$ locally symmetric space $M$ admits the following description \cite{Leuzinger}. There is an exhaustion function $h:M\rightarrow [0,+ \infty )$ such that for any $s\geq 0$,
$$M=M(s)\cup \amalg_{i=1}^k E_i(s),$$
where the sublevel set $M(s)=\{x\in M\mid h(x)\leq s \}$ is a compact submanifold and $E_1(s),...,E_k(s)$ are the disjoint cusp ends of $M\setminus M(s)$. Note that furthermore, each cusp end $E_i(s)$ is geodesically convex for any $s\geq 0$, i.e.  for any two points $x,y\in E_i(s)$, the unique geodesic between $x$ and $y$ is contained in $E_i(s)$. Also, $\pi^{-1}(E_i(s))$ is a disjoint union of horoballs.

Choose $b_0\in M(0)$ and, for every $1\leq i\leq k$ and $j \in \mathbb{N} $, points $b_{ij}\in E_i(j-1)\setminus E_i(j)$. Define a measurable map $\overline{p}:M\rightarrow
\{b_0\}\cup \{b_{ij}\}_{\{1\leq i \leq k,  j \in \mathbb{N} \}}$ by  $\overline{p}(M(0))=b_0$ and $\overline p(E_i(j-1)\setminus E_i(j))={b_{ij}}$
for every $i$ and $j\in \mathbb{N}$. Lift $\overline p$ to a $\Gamma$-equivariant measurable map
$$p:\widetilde{M}\longrightarrow \pi^{-1}(\{b_0\}\cup \{b_{ij}\}_{\{1\leq i \leq k, j\in \mathbb{N} \}} )\subset \widetilde{M}$$
with the property that the image by $p$ of a horoball of $\pi^{-1}(E_i(j))$ remains in the given horoball,
where $\pi$ is the  projection $\pi:\widetilde{M}\rightarrow M$.

The pullback map $p^*: C^q_{cpct}(M,\mathbb{R})\rightarrow C^q_{cpct}(M,\mathbb{R})$ is defined by sending a cochain $c\in  C^q_{cpct}(M,\mathbb{R})$ to the cochain
$$p^*(c)(\sigma)=c(\pi_* str(p(\sigma_0),...,p(\sigma_q))),$$
where $\sigma_0,...,\sigma_q\in \widetilde{M}$ are the vertices of a lift of $\sigma$ to $\widetilde{M}$. Note that $p^*(c)$ indeed has compact support. To see that, let $C$ be the compact support of $c$. There exists $s$ such that $C\subset M(s)$. Let $\sigma:\Delta^q\rightarrow M$ be a singular simplex whose image does not intersect $M(s)$, say $\mathrm{Im}(\sigma)\subset E_i(s)$. Any lift of $\sigma$ will be contained in a single horoball in the preimage $\pi^{-1}(E_i(s))$. The same is true for the vertices of this lift, and hence for the straightened simplex $str(p(\sigma_0),...,p(\sigma_q))$. It follows that the image of $\pi_* str(p(\sigma_0),...,p(\sigma_q))$ is contained in $E_i(s)$ so that the evaluation of $c$ on it indeed vanishes. It is obvious that $p^*$ commutes with the coboundary, and it is easy to check that it is homotopic to the identity. Thus it induces the identity on cohomology. Clearly, the cochain map restricts to bounded cochains and also induces the identity on the bounded cohomology group. Note that it is easy to show that these maps, at cohomology level, are independent of the chosen points $b_0$ and $b_{ij}$, though we are not going to need this fact.


The image of $p^*$ is contained in a subcocomplex which we denote by $C^q_{cpct,meas}(M,\mathbb{R})\subset C^q_{cpct}(M,\mathbb{R})$ and $C^q_{cpct,meas,b}(M,\mathbb{R})\subset $ $C^q_{cpct,b}(M,\mathbb{R})$ in the bounded case, of cochains whose evaluation on singular simplices $\sigma$ is given by evaluation on  the vertices of a lift on a $\Gamma$-equivariant measurable function $f:\widetilde{M}^{q+1}\rightarrow \mathbb{R}$.

A transfer map
$$trans:C^q_{cpct,meas}(M,\mathbb{R})\longrightarrow C(\widetilde{M}^{q+1},\mathbb{R})^G$$
is defined by sending a cochain $c$ to
$$trans(c)(x_0,...,x_q)=\int_{g\in D} p^*(c)(gx_0,...,gx_q)d\mu(g),$$
where $D$ is a fundamental domain for  $\Gamma\setminus G$ normalized to have measure $1$ and $x_0,...,x_q$ are points in $\widetilde{M}$.
To see that the integral is finite, let $d=\mathrm{max}_i d(x_0,x_i)$ and denote by $M(s)_d$ the closure of the $d$-neighborhood of $M(s)$.
Here, we choose a positive integer $s$ with $C\subset M(s)$. Since $M(s)_d$ is compact, so is
$$D_0=\{ g\in D \subset G \mid \pi(gx_0)\in M(s)_d\}.$$
We claim that $p^*(c)(gx_0,...,gx_q)=0$ for $g\in D\setminus D_0$. Indeed, since $gx_0\in \pi^{-1}(M\setminus M(s)_d)$, it belongs to a horoball of $\pi^{-1}(E_i(s))$ for some cusp $E_i$ and furthermore a ball of radius $d$ centered at $gx_0$ is also contained in the same horoball. In particular, $gx_1,...,gx_q$ and the straightened simplex $str(p(gx_0),...,p(gx_q))$ all belong to the same horoball and hence project to $E_i(s)$, on which $c$ vanishes. Note that for  $g\in D_0$, the evaluation $p^*(c)(gx_0,...,gx_q)$ takes only finitely many values (depending on $x_0,...,x_q$). It is further easy to see that the integral is independent of the choice of fundamental domain $D$, that $trans(c)$ is $G$-invariant and that the transfer map commutes with coboundaries. It thus induces a cohomology map and we will denote by
$$trans:H^q_{cpct}(M)\longrightarrow H^q_c(G).$$
 We denote by $trans_b$ the map on
bounded cohomology. Since $trans_b$ is, at cochain level, the
restriction of $trans$, the commutativity of the diagram
\[
\begin{CD}
  H^q_{cpct, b}(M)    @>trans_b>>          H^n_{c,b}(G,\br) \\
                      @VV c V                             @VV c V \\
  H^q_{cpct}(M) @>trans>>  H^n_{c}(G,\br)    \end{CD} \]
is obvious. Let us finally show that the full diagram (\ref{equ: CD with trans}) also commutes.

\begin{Prop} \label{prop: trans(omega)=omega}The diagram (\ref{equ: CD with trans}) commutes.
\end{Prop}

Note that in particular it follows that $trans(\omega_M)=\omega_{\widetilde{M}}$. Indeed, $\omega_M$ is represented
by $\rho \cdot (\Gamma\backslash \omega_{\widetilde M})$ where $\rho$ is a
compactly supported smooth function on $M$ so that $\int_M \rho \cdot
(\Gamma\backslash \omega_{\widetilde M})=vol(M)$. From section
\ref{deRham}, the image $trans(\omega_M)$ is independent of the
choice of $\rho$. Choose an exhausting sequence $(K_n)_{n\in \mathbb{N}}$ of compact connected subsets of $M$ with nonempty interior satisfying $\lim_{n\ra \infty} \mathrm{Vol}(K_n)=\mathrm{Vol}(M)$. Then, there exist $\epsilon_n>0$ and a smooth function $\rho_n : M \ra [0,1+\epsilon_n]$ supported on the $\epsilon_n$-neighborhood of $K_n$ such that $\rho_n |_{K_n} = 1+\epsilon_n$ and $\int_M \rho_n \cdot (\Gamma\backslash \omega_{\widetilde M})=vol(M)$. It is easy to see that the maps $\rho_n$ converges uniformly on any compact subset of $M$ to the constant function $c=1$ on $M$. Considering this sequence $(\rho_n)_{n\in \mathbb{N}}$, one can conclude that for any $\rho$, the average of a compactly supported differential form $\rho \cdot (\Gamma\backslash \omega_{\widetilde M})$ over the fundamental domain is $\omega_{\widetilde M}$.

\begin{proof} It remains to show that the lower diagram commutes. Note that the vertical isomorphisms are induced, at cochain level, by the map
$$\Phi: \Omega^q(G/K)^G \longrightarrow C((G/K)^{q+1},\br)^G$$
sending the differential form $\alpha$ to the cochain $\Phi(\alpha)$ mapping a $(q+1)$-tuple of points $(x_0,...,x_q)\in (G/K)^{q+1}$ to
$$\int_{ str(x_0,...,x_q)}\alpha.$$
This map is clearly $G$ (and hence $\Gamma$) - equivariant and
restricts to the subcomplexes of compact support. Furthermore, for
the left vertical arrow, $\Phi$ should further be precomposed with
the map sending a singular simplex in $\widetilde{M}$ to its
vertices. Here we use again the $\bq$-rank oneness to restrict to the subcomplexes of compact support, see section \ref{derham}. To prove the proposition, we need to show that $trans\circ
\Phi$ and $\Phi\circ trans_{dR}$ differ by a coboundary. One checks
easily that
\begin{eqnarray*}
 trans\circ \Phi(\alpha)(x_0,...,x_q)&=& \int_{g\in \Gamma\setminus G}\int_{str(p(gx_0),...,p(gx_q))} \alpha d\mu(g)
\end{eqnarray*}
while
\begin{eqnarray*}
 \Phi\circ trans_{dR}(\alpha)(x_0,...,x_q)&=& \int_{g\in \Gamma\setminus G}\int_{str(gx_0,...,gx_q))} \alpha d\mu(g) .
\end{eqnarray*}
It is immediate that, since $d\alpha=0$, the coboundary of the $G$-invariant cochain
$$(x_0,...,x_{q-1})\longmapsto \sum_{i=0}^{q-1} (-1)^i \int_{g\in \Gamma\setminus G}\int_{str(gx_0,...,gx_i,p(gx_i),...,p(gx_q))} \alpha d\mu(g)$$
is equal to the difference of the two given cocycles, which finishes
the proof of the proposition.
\end{proof}

\subsection{Norms and first inequality of Proposition \ref{mainProp}}\label{sec:5.3}

The inequality $ \| \omega_M \|^\infty \geq \| \omega_{\widetilde{M}} \|_\infty$ of Proposition \ref{mainProp} is an immediate consequence of the following lemma:

\begin{lemma} For every $\beta\in H^q_{cpct}(M)$, one has
$$ \| trans(\beta) \|_\infty \leq \| \beta \|^\infty.$$
\end{lemma}

\begin{proof}
Let $\Lambda<G$ be a cocompact torsion free lattice. Because of the existence of a transfer map, the restriction map $i^*:H^*_{c,b}(G)\rightarrow H^*_b(\Lambda)\cong H^*_b(\Lambda \setminus \widetilde{M})$ is isometric. Thus, for every $\alpha\in H^*_{c,b}(G)$, we have the equality $\|\alpha \|_\infty=\|i^*(\alpha)\|_\infty $. Furthermore since $\Lambda \setminus \widetilde{M}$ is compact, the upper and lower infinity norms agree $\|i^*(\alpha)\|_\infty=\|i^*(\alpha)\|^\infty$. As on a compact manifold a locally finite set of singular simplices is actually finite, it immediately follows that the norm of $\alpha \in H^*_{c,b}(G)$ can be computed as
$$\| \alpha \|_\infty = \sup_{\Phi \mathrm{\ finite}} \inf_{[a]=\alpha} \sup_{\sigma\in \Phi}|a(\sigma_0,...,\sigma_q)|,$$
where here and in the sequel, $\sigma_0,...,\sigma_q$ are the vertices of the singular simplex $\sigma$. Applying this equality to $\alpha=trans(\beta)$ and restricting the infimum to the cocycles of the form $a=trans(b)$, where $b\in C^n_{cpct}(M)$ is a compact support cocycle representing $\beta$, we obtain the inequality
$$\| trans(\beta)\|_\infty \leq  \sup_{\Phi \mathrm{\ finite}} \inf_{[b]=\beta} \sup_{\sigma\in \Phi}|trans(b)(\sigma)|.$$
Let $\Phi_p$ denote the following set of singular simplices,
$$\Phi_p=\{ \pi_* str(p(g\sigma_0),...,p(g\sigma_q)) \mid \sigma\in \Phi, \ g\in D\},$$
where $D$ is a fundamental domain for $\Gamma\setminus G$. Note that as $\Phi$ is finite $\Phi_p$ is locally finite. Since by definition, $trans(b)$ is an average of evaluations of $b$ on singular simplices in $\Phi_p$, it is immediate that
$$\sup_{\sigma\in \Phi} | trans(b)(\sigma_0,...,\sigma_q)|\leq \sup_{\sigma\in \Phi_p} |b(\sigma)|.$$
It follows that
$$\| trans(\beta)\|_\infty \leq \sup_{\Phi \mathrm{\ finite}} \inf_{[b]=\beta} \sup_{\sigma\in \Phi_p} |b(\sigma)| = \sup_{\Phi_p } \inf_{[b]=\beta} \sup_{\sigma\in \Phi_p} |b(\sigma)|,$$
where the first sup of the latter expression is taken over all families of simplices of the form $\Phi_p$, for a finite family $\Phi$. Finally, as $\Phi_p$ is locally finite we obtain
$$\| trans(\beta)\|_\infty\leq \sup_{\Psi \in \mathcal{S}^\mathrm{lf}_q} \inf_{[b]=\beta} \sup_{\sigma\in \Psi} |b(\sigma)|=\|\beta\|^\infty,$$
which finishes the proof of the lemma.
\end{proof}

\section{Cusp map and relation to bounded cohomology}\label{Section: coning}

\subsection{Cusp map}
To prove the theorem, we will need the existence of a
$\Gamma$-equivariant cusp map, which we establish in this section. We first review the reduction theory for arithmetic lattices. The main reference is \cite{Borel-Ji}.

Let $\mathbf{G}$ be a connected, semisimple algebraic group
defined over $\bq$ with trivial center. Let $X=G/K$ be the
associated symmetric space of noncompact type with
$G=\mathbf{G}(\br)^0$ and a maximal compact subgroup $K$ of $G$. A
closed subgroup $\mathbf{P}$ of $\mathbf{G}$ defined over $\bq$ is
called \emph{rational parabolic subgroup} if $\mathbf{P}$ contains
a maximal, connected, solvable subgroup of $\mathbf{G}$. For any
rational parabolic subgroup $\mathbf{P}$ of $\mathbf{G}$, we
obtain the \emph{rational Langlands decomposition} of
$P=\mathbf{P}(\br)$:
$$P = N_{\mathbf{P}} \times A_\mathbf{P} \times M_\mathbf{P},$$
where $N_\mathbf{P}$ is the real locus of the unipotent radical $\mathbf{N}_\mathbf{P}$ of $\mathbf{P}$,
$A_\mathbf{P}$ is a stable lift of the identity component of the real locus of the the maximal $\bq$-torus in the Levi quotient $\mathbf{P} / \mathbf{N}_\mathbf{P}$ and $M_\mathbf{P}$ is a stable lift of the real locus of the complement of the maximal $\bq$-torus in $\mathbf{P} / \mathbf{N}_\mathbf{P}$.

Write $X_\mathbf{P}=M_\mathbf{P}/K \cap M_\mathbf{P}$. Let us denote by $\tau :  M_\mathbf{P} \ra
X_\mathbf{P}$ the canonical projection. After fixing a basepoint $x_o \in X$, we have an analytic diffeomorphism
$$\mu : N_{\mathbf{P}} \times A_\mathbf{P} \times X_\mathbf{P} \ra X, \ (n, a, \tau(m)) \ra nam \cdot x_0,$$
This is called the \emph{rational horocyclic decomposition} of $X$.

Let $\Gamma$ be an arithmetic lattice in $G=\mathbf{G}(\br)$ with $\bq$-$\mathrm{rank}(\mathbf{G})=1$. It is a well known fact due to A. Borel and Harish-Chandra that there are only finitely many $\Gamma$-conjugacy classes of proper rational parabolic subgroups. Furthermore, all rational proper parabolic subgroups are minimal and these subgroups are conjugate under $\mathbf{G}(\bq)$ (see \cite[Theorem 11.4]{Borel}). Note that the set of simple positive $\bq$-roots contains only a single element and $\mathrm{dim} A_\mathbf{P}=1$ for any proper rational parabolic subgroup $\mathbf{P}$ of $\mathbf{G}$.

The locally symmetric space $M=\Gamma \backslash X$ has finitely many cusps and each cusp corresponds to a $\Gamma$-conjugacy class of a minimal rational parabolic subgroup. Let $\mathbf{P}_1,\ldots,\mathbf{P}_k$ denote representatives of the $\Gamma$-conjugacy classes. For each minimal rational parabolic subgroup $\mathbf{P}_i$, set $$V_{\mathbf{P}_i}(t) = \mu (N_{\mathbf{P}_i} \times A_{\mathbf{P}_i,t} \times X_{\mathbf{P}_i})$$ where $A_{\mathbf{P}_i,t}=\{ a\in A_{\mathbf{P}_i} : \alpha_i(\log a) > t\}$. Here $\alpha_i$ is the unique positive simple $\bq$-root corresponding to $\mathbf{P}_i$. Write $\Gamma_\mathbf{P_i}=\Gamma \cap \mathbf{P_i}$ for $i=1,\ldots,k$. There exists $t_0>0$ such that for all $t>t_0$, $M$ admits the following disjoint decomposition:
$$M=\Omega_t \cup  \amalg_{i=1}^k \Gamma_{\mathbf{P}_i} \backslash V_{\mathbf{P}_i}(t),$$
where $\Omega_t$ is a compact submanifold with boundary.
Furthermore, for $t>t_0$ and each $i$, we have
$$\Gamma_{\mathbf{P}_i} \backslash V_{\mathbf{P}_i}(t) \cong \Gamma_{\mathbf{P}_i} \backslash N_{\mathbf{P}_i} \times A_{\mathbf{P}_i,t} \times X_{\mathbf{P}_i} = \Gamma_{\mathbf{P}_i} \backslash ( N_{\mathbf{P}_i} \times X_{\mathbf{P}_i}) \times A_{\mathbf{P}_i,t}.$$
Geometrically, each $V_{\mathbf{P}_i}(t)$ is a horoball in $X$. In fact, the disjoint decomposition $M=\Omega_t \cup  \amalg_{i=1}^k \Gamma_{\mathbf{P}_i} \backslash V_{\mathbf{P}_i}(t)$ gives the same disjoint decomposition $M=M(s) \cup  \amalg_{i=1}^k E_i(s)$ for some $s>0$  in Section \ref{subsec: def of p}.

\begin{Prop}  \label{prop: coning map} Let $M$ be a $\bq$-rank $1$ locally
symmetric space and let $\Gamma=\pi_1(M)$ be its fundamental group.
Let $G$ denote the identity component of the isometry group of the universal cover
$\widetilde{M}$ of $M$ and let $P$ be a minimal rational parabolic subgroup of $G$. Then
there exists a $\Gamma$-equivariant measurable cusp map
$$ f:\widetilde{M} \longrightarrow G/P.$$
such that $f$ is constant on each horoball in $\widetilde{M}$ corresponding to a cusp of $M$.
\end{Prop}

\begin{proof}
Fix a $t>t_0$. Decompose $M=\Omega_t \cup \amalg_{i=1}^k \Gamma_{\mathbf{P}_i} \backslash V_{\mathbf{P}_i}(t)$ as before. We set $\mathbf{P}=\mathbf{P}_1$ and $P=\mathbf{P}(\br)^0$.
Since all minimal rational parabolic subgroups are conjugate under $\mathbf{G}(\bq)$, there exists $g_i \in G$ with $\mathbf{P}_i=g_i \mathbf{P} g_i^{-1}$ for $i=1,\ldots,k$. Let $F_{\mathbf{P}_i}(t)$ be a fundmental domain for the action of $\Gamma_{\mathbf{P}_i}$ on $V_{\mathbf{P}_i}(t)$ and $F_t$ be a fundamental domain for the action of $\Gamma$ on $\pi^{-1}(\Omega_t)$ where $\pi : \widetilde{M} \ra M$ is the universal covering map of $M$.
Map $F_t$ to $P$ and $F_{\mathbf{P}_i}(t)$ to $g_iP$. Extending this map $\Gamma$-equivariantly, we get a $\Gamma$-equivariant measurable map $f : \widetilde{M} \ra G/P$. Then it is easy to see that $f(x)=g_iP$ for all $x\in V_{\mathbf{P}_i}(t)$. This completes the proof.
\end{proof}

\subsection{The induced bounded cohomology map}

\begin{thm} \label{thm: induced from coning} Let $M$ be a $\bq$-rank $1$ locally symmetric space and let $\Gamma=\pi_1(M)$ be its fundamental group. Suppose that the fundamental group of any cusp of $M$ is amenable. Then the $\Gamma$-equivariant cusp map from Proposition \ref{prop: coning map}
$$  f:\widetilde{M} \longrightarrow G/P$$
induces for $q\geq 1$ a cohomology map
$$ f^*: H^q_{b}(\Gamma)\longrightarrow H^q_{cpct, b}(M)$$
such that
$$ trans_b \circ f^*=trans_\Gamma$$
and
$$ \| f^*(\alpha)\|_\infty \leq \| \alpha \|_\infty$$
for any $\alpha \in H^q_{b}(\Gamma)$. 
\end{thm}

\begin{proof} It is shown in \cite[Section 6]{BucherMonod} that the bounded cohomology group $H^*_b(\Gamma)$ can be computed isometrically on the cochain complex of alternating $\Gamma$-invariant bounded functions $\ell^\infty_\mathrm{alt}((G/P)^{q+1},\mathbb{R})^\Gamma$ endowed with its homogeneous coboundary operator. We will use this cochain complex to induce the cohomology map $f^*:H^*_b(\Gamma)\rightarrow H^*_{cpct,b}(M)$ in degree $\geq 1$. Let thus $  f:\widetilde{M} \longrightarrow G/P$ be the cusp map. It
induces in degree $\geq 1$ a cochain map
$$f^*:\ell_{alt}^\infty((G/P)^{q+1},\mathbb{R})^\Gamma  \longrightarrow C^q_{cpct,b}(\widetilde{M})^\Gamma \cong C^q_{cpct,b}(M)$$
by sending a  bounded, $\Gamma$-invariant alternating cochain $c:(G/P)^{q+1}\rightarrow \mathbb{R}$ to the cochain $f^*(c)$ defined by mapping a singular simplex $\sigma:\Delta^q\rightarrow \widetilde{M}$ to $c(f(x_0),...,f(x_q))$, where $x_0,...,x_q\in \widetilde{M}$ are the vertices  of $\sigma$. The cochain $f^*(c)$ is clearly $\Gamma$-equivariant and bounded. Furthermore, it has compact support contained in $M(0)$ since if the image of $\pi(\sigma)$ does not intersect $M(0)$, then it is contained in a cusp $E_i$, so that $\sigma$ is contained in one horoball, and its vertices $x_0,...,x_q$ all project to the same boundary point $f(x_0)=...=f(x_q)\in G/P$. Since the cochain $c$ is alternating, $f^*(c)(\sigma)=0$. It is for this last assertion that we use the fact that we are in degree $\geq 1$.

It is immediate that $f^*$ commutes with the coboundary operator $\delta$ and hence induces a cohomology map
$$ f^*: H^*_{b}(\Gamma)\longrightarrow H^*_{cpct, b}(M).$$
Since the cochain map $f^*$ does not augment norm, the last assertion of the theorem is obvious. It remains to see that $ trans \circ f^*=trans_\Gamma$.

Consider the following diagram, where the cochain map $p^*$ is as defined in the beginning of Subsection \ref{subsec: def of p}:
\begin{equation*}
\xymatrix{ \ell_{alt}^\infty((G/P)^{q+1},\mathbb{R})^\Gamma  \ar@/^2pc/[rr]^{f^*} \ar[r]^{\ \ \ \ \ \  f^*  }&C^q_{cpct,b}(M) \ar[r]^-{p^* }  \ar@<-.7ex>[rd]_{trans } &  C^q_b(G/K)^\Gamma \ar[d]^{trans_\Gamma} \\
 & & C^q_b(G/K)^G.}
\end{equation*}
Observe that the composition $p^*\circ f^*$ is equal to $f^*$. Since the cochain map $f^*: \ell_{alt}^\infty((G/P)^{*+1})^\Gamma\rightarrow C^q_b(G/K)^\Gamma $ extends the identity $\mathbb{R}\rightarrow \mathbb{R}$ in degree $-1$ it induces the identity on cohomology. Thus, on cohomology, the
diagram becomes
\begin{equation*}
\xymatrix{H^q_b(\Gamma) \ar@/^2pc/[rr]^{f^*=Id} \ar[r]^{f^*  }&H^q_{cpct,b}(M) \ar[r]^{p^* \ \ \ }  \ar@<-.7ex>[rd]_{trans } &  H^q_b(\Gamma) \ar[d]^{trans_\Gamma} \\
 & & H^q_b(G/K)^G.}
\end{equation*}
It is now clear that $trans_\Gamma= trans_b \circ f^*$, which finishes the proof of the theorem. \end{proof}


\subsection{Proof of the second inequality of Proposition \ref{mainProp}}
We can now prove that $ ||\omega_M||_\infty \leq ||\omega_{\widetilde{M}}||_\infty$. Consider the following diagram:

\begin{equation*}
\xymatrix{ H^n_{c,b}(G,\br)  \ar@<+.7ex>@{^{(}->}[r]^{i^*} \ar[d]^c & H^n_{b}(\Gamma,\br) \ar@<+.7ex>@{->>}[l]^{trans_\Gamma}\ar[r]^{f^* \ \ \ }  & H^n_{cpct, b}(M,\br)  \ar@/_2pc/[ll]^{trans_b} \ar[d]^c \\
H^n_{c}(G,\br)  & & \ar[ll]_{trans} H^n_{cpct}(M,\br). }
\end{equation*}
Using that $trans$ commutes with the comparison map $c$, that $trans_b \circ f^*=trans_\Gamma$ (Theorem \ref{thm: induced from coning}), and that $trans_\Gamma \circ i^*$ is the identity on $H^n_{c,b}(G,\br)$, we show that
$$ trans \circ c \circ f^* \circ i^* = c. \label{equ: comm diag}$$
Indeed,
\begin{eqnarray*}
\underbrace{trans \circ c}_{=c\circ trans_b} \circ f^*\circ i^* &=& c\circ {\underbrace{trans_b \circ f^* }_{=trans_\Gamma}}\circ i^*\\
&=& c\circ \underbrace{trans_\Gamma \circ i^*}_{=\mathrm{Id}} =c.
\end{eqnarray*}
It remains to establish the desired inequality:
\begin{eqnarray*}
\| \omega_M\|_\infty&=& \mathrm{inf}\{ \| \alpha \|_\infty\mid \alpha\in H^n_{cpct, b}(M,\br),\  c(\alpha)=\omega_M \}\\
&\leq & \mathrm{inf}\{ \| \alpha \|_\infty\mid \alpha\in \mathrm{Im}(f^* \circ i^*),\  trans\circ c(\alpha)=trans(\omega_M)=\omega_{\widetilde{M}} \}\\
&=&\mathrm{inf}\{ \underbrace{\| (f^* \circ i^*)(\beta) \|_\infty}_{\leq \|\beta \|_\infty} \mid \beta\in H^n_{c,b}(G,\br),\  \underbrace{trans\circ c\circ f^* \circ i^*}_{=c}(\beta)=\omega_{\widetilde{M}} \}\\
&=& \| \omega_{\widetilde{M}} \|_\infty,
\end{eqnarray*}
where we have used the facts that the map $trans:H^n_{cpct}(M,\br)\rightarrow H^n_c(G,\br)$ is injective (even bijective since $trans(\omega_M)=\omega_{\widetilde{M}}$) in top degree and that $f^*$ and $i^*$ do not increase norms. This finishes the proof of the second inequality of Proposition \ref{mainProp}.

\section{The product of $\br$-rank $1$ symmetric spaces}\label{section: product of rank 1}

In this section, we establish the proportionality principle for the simplicial volume of certain $\mathbb{Q}$-rank $1$ locally symmetric spaces covered by a product of $\br$-rank $1$ symmetric spaces without any assumption on their ends.

Let $H$ be the full isometry group of $\widetilde{M}$. An action of $H$ on $\br$ is defined as follows: An element $h$ of $H$ acts by multiplication by $+1$ (resp. $-1$) if $h$ preserves (resp. reverses) the orientation in $\widetilde{M}$. We denote by $\br_\varepsilon$ the Banach space $\br$ endowed with the action of $H$. The continuous $H^n_{c}(H,\br_\varepsilon)$ and continuous, bounded $H^n_{c,b}(H,\br_\varepsilon)$ cohomology of $H$ with coefficients in $\br_\varepsilon$ is defined by replacing the $H$-invariant cochains by $H$-equivariant ones in the corresponding definitions.


\begin{proof}[Proof of Theorem \ref{thm:1.4}]
The simplicial volume $\| M \|_\mathrm{lf}$ of $M$ can be reformulated in terms of $\ell^1$-homology as follows:
\begin{eqnarray}\label{eqn:4} \| M \|_\mathrm{lf} =\inf_{\alpha \in [M]^{\ell^1}} \sup \left\{ \frac{1}{\| \varphi \|_\infty} \ \bigg| \ \varphi \in H^n_b(M) \text{ and } \langle \varphi, \alpha \rangle =1 \right\}, \end{eqnarray}
where $[M]^{\ell^1}$ is the set of all $\ell^1$-homology classes that are represented by at least one locally finite cycle with finite $\ell^1$-norm. For a proof, see \cite{Loeh}. For the reader's convenience, we sketch a proof that the Kronecker product $\langle \cdot, \cdot \rangle : H^*_b(M) \otimes H^{\ell^1}_*(M) \rightarrow \br$ is well defined. In fact, this follows from a well-known fact in functional analysis that $C_b^*(M)$ is the topological dual space of the $\ell^1$-completion $C^{\ell^1}_*(M)$ of $C_*(M)$. In other words, a linear functional on $C^{\ell^1}_*(M)$ is continous if and only if it is bounded. Let $f : C_*(M) \rightarrow \br$ be a bounded cochain in $C_b^*(M)$. Denote by $f' :C^{\ell^1}_*(M) \rightarrow \br$ the unique continuous extension of $f$. One can define a map $\langle \cdot,\cdot \rangle : C_b^*(M) \otimes C^{\ell^1}_*(M) \rightarrow \br$ via the evaluation map between $C^{\ell^1}_*(M)$ and its dual cochain complex by
$$\langle f, \cdot \rangle =f'(\cdot),$$
Noting that $\delta f (c) = f(\partial c)$ for $c \in C_{*+1}(M)$, it is easy to see that $(\delta f)'(c')=f'(\partial c')$ for $c' \in C^{\ell^1}_{*+1}(M)$. Hence, we have
$$\langle \delta f, c' \rangle = (\delta f)'(c')= f'(\partial c') = \langle f, \partial c' \rangle.$$
This implies that the Kronecker product between $H^*_b(M)$ and $H^{\ell^1}_*(M)$ is well defined.

Define a $H$-invariant cocycle $\Theta: \widetilde{M}^{n+1} \rightarrow \br_\varepsilon$ by
$$\Theta(x_0,\ldots,x_n) = \int_{[x_0,\ldots,x_n]} \omega_{\widetilde{M}}$$
where $[x_0,\ldots,x_n]$ is the geodesic simplex in $\widetilde{M}$ with an ordered vertex set $\{x_0,\ldots,x_n\}$ and $\omega_{\widetilde{M}}$ is the $H$-invariant volume form on $\widetilde{M}$. Note that the cocycle $\Theta$ is bounded because the volume of top dimensional geodesic simplices in a product of $\br$ rank $1$ symmetric spaces is uniformly bounded from above.
Hence, $\Theta$ determines a continuous cohomology class $[\Theta]^H \in H^n_c(H,\br_\varepsilon)$ and a continuous bounded cohomology class $[\Theta]^H_b$ in $H^n_{c,b}(H,\br_\varepsilon)$ with $c([\Theta]^H_b)=[\Theta]^H$.

Let $G$ be the connected component of the identity of $H$. Note that $G$ is a finite index subgroup of $H$.
Hence, the inclusion $G \subset H$ induces isometric embeddings $res^H : H^*_c(H,\br_\varepsilon) \rightarrow H^*_c(G,\br)$ and $res^H_b : H^*_{c,b}(H,\br_\varepsilon) \rightarrow H^*_{c,b}(G,\br)$ which are realized by the canonical inclusions of cocomplexes $C^*_c(\widetilde{M},\br_\varepsilon)^H \subset C^*_c(\widetilde{M},\br)^G$ and $C^*_{c,b}(\widetilde{M},\br_\varepsilon)^H \subset C^*_{c,b}(\widetilde{M},\br)^G$. Similarly, since $\Gamma$ is a lattice in $G$, the inclusion $\Gamma \subset G$ induces an isometric embedding $res^G_b : H^*_{c,b}(G,\br) \rightarrow H^*_b(\Gamma,\br)$ where $\Gamma$ is the fundamental group of $M$.

It is clear that $res^H([\Theta]^H)=\omega_{\widetilde{M}}$. Since the comparison map is an isomorphism and $res^H$, $res^H_b$ and $res^G_b$ are isometric embeddings,
$$\| (res^G_b \circ res^H_b )([\Theta]^H_b) \|_\infty =\| [\Theta]^H_b \|_\infty = \| [\Theta]^H \|_\infty = \|\omega_{\widetilde{M}} \|_\infty.$$

Clearly, the cocycle $\Theta$ represents $(res^G_b \circ res^H_b )([\Theta]^H_b) \in H^n_b(M,\br)$.
Let $\alpha \in [M]^{\ell^1}$. Then there is a locally finite fundamental cycle $c$ with $\| c\|_1 <\infty$ such that $c$ represents $\alpha$ in the $\ell^1$-homology of $M$. For $\mathbb{Q}$-rank $1$ locally symmetric spaces, the geodesic straightening map on the locally finite chain complex is well defined and chain homotopic to the identity \cite{Kim-Kim}. Hence, the geodesic straightening $str(c)$ of $c$ is again a locally finite fundamental cycle with $\|str(c) \|_1 \leq \|c \|_1 < \infty$ and we have $$str(c)-c=\partial H_n^\mathrm{lf}(c) + H_{n-1}^\mathrm{lf} \partial (c) = \partial H_n^\mathrm{lf}(c)$$
where $H_*^\mathrm{lf} : C_*^\mathrm{lf}(M,\br) \rightarrow  C_{*+1}^\mathrm{lf}(M,\br)$ is a chain homotopy between the geodesic straightening map and the identity that is constructed by the canonical straight line homotopy. It is easy to see that $$\| H_n^\mathrm{lf}(c) \|_1 \leq (n+1) \|c\|_1 <\infty $$
and thus, $str(c)$ and $c$ represent the same $\ell^1$-homology class $\alpha$.
Furthermore, note that
\begin{eqnarray}\label{eqn:6} \langle \Theta, str(c) \rangle = \mathrm{Vol}(M).\end{eqnarray}
We refer the reader to \cite{Kim-Kim} for more details.

Equation (\ref{eqn:6}) implies that for any $\alpha \in [M]^{\ell^1}$,
$$\left\langle \frac{(res^G_b \circ res^H_b )([\Theta]^H_b)}{\mathrm{Vol}(M)}, \alpha \right\rangle = \frac{\langle \Theta, str(c) \rangle}{\mathrm{Vol}(M)}=1.$$
Hence, in the view of Equation (\ref{eqn:4}), we have
$$  \| M \|_\mathrm{lf} = \frac{\mathrm{Vol}(M)}{\| \omega_M \|^\infty} \geq \frac{\mathrm{Vol}(M)}{\| (res^G_b \circ res^H_b )([\Theta]^H_b) \|_\infty}=\frac
{\mathrm{Vol}(M)}{\| \omega_{\widetilde{M}} \|_\infty}.$$
Thus, we obtain an inequality $ \| \omega_M \|^\infty \leq \| \omega_{\widetilde{M}} \|_\infty$.
From the opposite inequality in Section \ref{sec:5.3}, we finally have
$$ \| \omega_M \|^\infty = \| \omega_{\widetilde{M}} \|_\infty,$$
which completes the proof.
\end{proof}

\begin{remark} Theorem \ref{thm:1.4} holds for reducible $\mathbb{Q}$-rank $1$ locally symmetric spaces.
Note that the fundamental groups of the ends of reducible $\mathbb{Q}$-rank $1$ locally symmetric spaces are not amenable.
Hence, this strongly supports the existence of the proportionality principle for the simplicial volume of arbitrary $\mathbb{Q}$-rank $1$ locally symmetric spaces. \end{remark}

\end{document}